\documentclass[a4paper,12pt]{article}
\pdfoutput=1
\usepackage[english]{babel}
\usepackage[margin=1in]{geometry}

\usepackage{amsmath,amssymb,amsthm}

\usepackage{natbib}

\usepackage{ifpdf}
\ifpdf
\usepackage[bookmarks=false,%
pdfstartview=FitH,linkbordercolor={0.5 1 1},%
citebordercolor={0.5 1 0.5},unicode,%
hyperindex]{hyperref}%
\else
\usepackage{breakurl}                          
\fi

\newtheorem{theorem}{Theorem}

\newtheorem{lemma}{Lemma}

\usepackage{setspace}

\begin{document}

\thispagestyle{empty}
\vskip2cm
\begin{center}
{\bf {\Large Survival Analysis of Particle Populations in Branching Random Walks}}
\end{center}

\begin{center}
{\large Anastasiia Rytova
and Elena Yarovaya}\footnote{email:
yarovaya@mech.math.msu.su, address: Department of  Mechanics and Mathematics, Moscow State
University, Moscow 119991, Russia}
\end{center}
\begin{center}
{\it Department of Probability Theory, Faculty of Mathematics and Mechanics, \\Lomonosov Moscow State
University,  Moscow, Russia}
\end{center}
\vskip3mm

\begin{abstract}
It is a common practice to describe branching random walks in
terms of birth, death and walk of particles, which makes it easier to use them
in different applications. The main results obtained for the models
of symmetric continuous-time heavy-tailed branching random walks on a multidimensional
lattice.
We will be mainly interested in studying the problems related to the limiting behavior
of branching random walks such as the existence of phase transitions under change
of various parameters, the properties of the limiting distribution and the survival
probability of a particle population.
Emphasis is made on the survival analysis.
The answers to these and other questions essentially
depend on numerous factors which affect the properties of a branching random
walk. Therefore, we will try to describe how the properties of a branching
random walk depend on such characteristics of the underlying random walk
as finiteness or infiniteness of the variance of jumps.
The presented results
are based on the Green's function representations of the transition probabilities of a
branching random walk.
\medskip

\medskip

\end{abstract}

\section{Introduction}
It is a common practice to describe branching random walks (BRWs) in terms of birth, death and walk of particles, which makes it easier to use them in different areas of nature sciences: statistical physics \citep{ZMRS88:e}, chemical kinetics \citep{GarMol90:e}, the theory of homopolymers \citep{CKMV09:e}, population dynamic studies \citep{BMW16:e}.
The use
of BRWs on multidimensional lattices  with one centre of particle generation and a finite variance of random walk jumps in the reliability theory was discussed in \citep{Y10-AiT:e}. In the present paper,  the models of symmetric continuous-time BRWs on the $d$-dimensional lattice ${\mathbf{Z}}^{d}$, $d\geq 1$, with a few sources of particle birth and death at lattice points, called branching sources, are reviewed. Emphasis is made on the survival analysis of the particle population on  ${\mathbf{Z}}^{d}$ for a BRW  with one branching source.

We will be mainly interested in studying the problems related to the limiting behavior of BRWs such as the existence of phase transitions under change of various parameters, the properties of the limiting distribution and the survival probability of a particle population.
The answers to these and other questions essentially depend on numerous factors which affect the properties of a BRW. Therefore, we will try to describe, rather detailed, how the properties of a BRW depend on such characteristics of the underlying random walk as finiteness or infiniteness of the variance of jumps.
The presented results are based on the Green's function representations  of the transition probabilities of the underlying branching walk.

Introduction is followed by a formal description of the model of
a symmetric BRW with a finite number of branching sources described in Section~\ref{sect:Model}. Moreover,  in Section~\ref{sect:Model} we formulate the recurrence criteria for BRWs with a finite and infinite variance of random walk jumps in terms of Green's functions and remind some recent results on the symmetric BRWs. In Section~\ref{sect:SA} we represent equations for the probability of the presence of particles
at the origin and the survival probability of a particle population without any assumptions on the variance of jumps of the underlying random walk and discuss very shorty their possible applications to the reliability theory.
In Section~\ref{sect:SPHTBRW} main theorems whose proofs are based on the results of~Section~\ref{sect:SA} are obtained. In Section~\ref{sect:Conclusion} we give some remarks on the asymptotic behavior of survival probabilities for BRWs with a finite variance of jumps for comparison with the case of  heavy-tailed BRWs.

\section{Model and Previous Results}
\label{sect:Model}
The evolution of the particle system on ${\mathbf{Z}}^{d}$ is described by the
number of particles at time $t$ at each point $y\in {\mathbf{Z}}^{d}$ under
the assumption that at the time $t=0$ the system consists of one particle
located at the point $x$.  The particle walks on ${\mathbf{Z}}^{d}$ until it reaches
one of the  points $x_{1},x_{2},\dots,x_{N}$, $N<\infty$,
 where it can die or produce a random
number of offsprings.  It is assumed that evolution of the newborn
particles obeys the same law independently of the rest of the particles and
the prehistory.  Now we proceed to a full description of the model.

The random walk of particles is defined by an infinitesimal
transition matrix $A=\Vert a(x,y)\Vert_{x,y \in {\mathbf{Z}}^{d}}$ and
is assumed to be symmetric, $a(x,y)=a(y,x)$; homogeneous,
$a(x,y)=a(0,y-x)=a(y-x)$; irreducible, that is, for every $z\in \mathbf{Z^{d}}$ there exists a set of vectors $z_{1},z_{2},\dots,z_{k}\in \mathbf{Z^{d}}$ such that $z=\sum_{i=1}^{k}z_{i}$ and $a(z_{i})\neq0$ for $i=1,2,\dots,k$;
regular, $\sum_{x \in {\mathbf{Z}}^{d}}a(x)=0$, where $a(x)\geq 0$ for $x\neq 0$ and $-\infty<a(0)<0$.  In virtue of symmetry and homogeneity
of the random walk, the conditions $\sum_{y\in {\mathbf{Z}}^{d}}
a(x,y)=\sum_{x\in {\mathbf{Z}}^{d}} a(x,y)=0$ are satisfied for the
matrix $A$.

At some points $x_{1},x_{2},\ldots,x_{N}$, called
\emph{branching sources}, every particle can die or give offsprings. The reproduction law
at the source $x_{i}$, $i=1,2,\dots,N$, is defined by the continuous-time Bienaym\'e-Galton-Watson branching process \citep{Se2:e,AN:e},
by the following infinitesimal generation function
\[
f(u,x_{i})=\sum_{n=0}^\infty b_n(x_{i}) u^n,\quad 0\le u \le1,
\]
where $b_n(x_{i})\ge0$ for $n\ne 1$, $b_1(x_{i})<0$ and $\sum_{n} b_n(x_{i})=0$.
We assume  $f^{(r)}(1,x_{i})<\infty$ for every $r\in{\mathbf{N}}$. In the future,
an important role will play the values
\[
\beta_{i}=f'(1,x_{i})=\sum_{n}nb_{n}(x_{i}) =
(-b_{1}(x_{i}))\left(\sum_{n\neq1}n\frac{b_{n}(x_{i})}{(-b_{1}(x_{i}))}-1\right),
\]
for $i=1,2,\dots,N$, called \emph{intensity} of the branching source  $x_{i}$, where the last sum is the mean number of offsprings  born at the point
$x_{i}$.

By $p(t,x,y)$ we denote the transition probability of the underlying random walk. This
function is implicitly determined by the transition intensities $a(x,y)$
(see, e.g., \citep{GS:e,YarBRW:e}). Then Green's function of the
operator $\mathcal{A}$ corresponding to the matrix $A$ (see detailed definition in \citep{Y18-TPA:e}) can be represented as the Laplace transform of the
transition probability $p(t,x,y)$:
\[
G_\lambda(x,y):=\int_0^\infty e^{-\lambda t}p(t,x,y)\,dt,\quad \lambda\geq 0.
\]

The analysis of BRWs essentially depends on whether the value of
$G_{0}:=G_{0}(0,0)$ is finite or infinite. As is known, see, e.g., \citep{Sp69:e}, a random walk is \emph{transient} if $G_{0}(0,0)<\infty$ and
\emph{recurrent} if $G_{0}(0,0)=\infty$. We generalize this definition on  BRWs. A~BRW is called \emph{transient} if the underlying random walk is \emph{transient} and \emph{recurrent}  if the underlying random walk is \emph{recurrent}. If the variance of jumps of the underlying random walk is finite,
that is,
\begin{equation}\label{E:findisp}
\sum_{z\in \mathbf{Z}^{d}} |z|^2 a(z)<\infty,
\end{equation}
where $|z|$ is the Euclidian norm  of the vector $z$, then we get the following \emph{recurrence criteria} for BRWs with \emph{a finite variance} of jumps: $G_{0}=\infty$ for
$d=1,2$, and $G_{0}<\infty$ for $d\ge 3$ (see, e.g., \citep{YarBRW:e}) followed from the asymptotic relation $p(t,x,y)\sim \gamma_{d}t^{-\frac{d}{2}}$ as $t\to \infty$, where $\gamma_{d}$ is a positive constant depending on dimension of ${\mathbf{Z}}^{d}$.

Now we consider the case with an another assumption on the transition intensities.
We suppose that for all $z\in\mathbf{Z}^{d}$ with sufficiently large norm $|z|$ the asymptotic
relation
\begin{equation}\label{E:infindisp}
a(z)\sim\frac{H\left(\frac{z}{|z|}\right)}{|z|^{d+\alpha}},\quad\alpha\in(0,2),
\end{equation}
holds, where $H(\cdot)$ is a continuous positive symmetric function  on the
sphere
$\mathbf{S}^{d-1}=\{z\in\mathbf{R}^{d} : |z|=1\}$.
Condition (\ref{E:infindisp}), unlike (\ref{E:findisp}),
leads to the divergence of the series in~\eqref{E:findisp} and thereby
to infinity of the variance of jumps. Under assumption~\eqref{E:infindisp} a BRW is called \emph{a heavy-tailed BRW}. In this case we get the following \emph{recurrence criteria} for BRWs  with \emph{infinite variance} of jumps:
$G_{0}=\infty$ for $d=1$, $\alpha \in [1,2)$, and $G_{0}$  is finite if $d
= 1$, $\alpha \in (0,1)$, or $d\geq 2$, $\alpha \in (0,2)$,
followed from the behavior of $p(t,x,y)\sim h_{d,\alpha}t^{-\frac{d}{\alpha}}$ as $t\to \infty$, where $h_{d,\alpha}$ is a positive constant depending on dimension of ${\mathbf{Z}}^{d}$ (see, details in \citep{RY16:MN,Y13-CommStat:e}).

Put
\[
\beta_{c}:=\frac{1}{G_{0}(0,0)}.
\]
Then $\beta_{c}=0$ for $G_{0}=\infty$ and $\beta_{c}>0$ for $G_{0}<\infty$.

Let $\mu_{t}(y)$ be the number of particles at an arbitrary lattice point $y\in\mathbf{Z}^d$  and  $\mu_{t}$ be the number of particles on the entire lattice, called \emph{the particle population}, at time $t$. The
moments of numbers of particles $\mu_{t}(y)$ and  $\mu_{t}$  are denoted, respectively, by $m_n(t,x,y):={\mathnormal{E}}_x
\mu^n_t(y)$ and $m_n(t,x):={\mathnormal{E}}_x \mu^n_t$, $n\in
{\mathbf{N}}$, where ${\mathnormal{E}}_x$ is the conditional expectation
provided that $\mu_0(\cdot)=\delta_x(\cdot)$.  The temporal asymptotic behavior of
the moments for one branching source under the condition~(\ref{E:findisp}) was studied in details in \citep{BY-1:e,YarBRW:e}.

The value $\beta_{c}$ is
critical in the sense that the asymptotic behavior of $\mu_{t}(y)$ and $\mu_{t}$ is
different for $\beta>\beta_{c}$, $\beta=\beta_{c}$ and
$\beta<\beta_{c}$.  Additionally, in some publications, of which we mention
\citep{ABY98-1:e,BY-1:e,YarBRW:e,Y09DM:e}, it was established the difference in the limit
behavior for $t\to \infty$ not only of the moments but also of the numbers of particles
themselves.  Therefore, the following definition of the BRW criticality makes
sense.  A BRW is referred to as \emph{supercritical} if $\beta>\beta_{c}$, \emph{critical} if
$\beta=\beta_{c}$, and \textit{subcritical} if $\beta<\beta_{c}$.

Recall known results, see, e.g.~\citep{YarBRW:e}, for a BRW with one branching source that will be needed in Section~\ref{sect:SPHTBRW}.  If $\beta>\beta_{c}$, then the equation
\[
G_{\lambda}(0, 0) =\frac{1}{\beta}
\]
has a single positive solution $\lambda_{0}$, whence it follows that the
random variables $\mu_{t}(y)$ and $\mu_{t}$ have limit distributions as
$t\to \infty$ \citep{BY-1:e,YarBRW:e}, in the sense of convergence of the moments
under the normalization $e^{-\lambda_{0}t}$.
In \citep{Y16-MCAP:e} it was established for BRWs with $N$ branching sources of equal intensities $\beta$ that, under each of the conditions~\eqref{E:findisp} and~\eqref{E:infindisp}, if $G_0=\infty$ then $\beta_c=0$ for $N\geq1$, if $G_0<\infty$ then $\beta_c=G_0^{-1}$ for $N=1$ and $0<\beta_c<G_0^{-1}$ for $N\geq2$. Note that for heavy-tailed BRWs the critical value $\beta_c$ depends on the parameter $\alpha$. In \citep{KhristolubovYarovaya} the asymptotics of the moments of the particle population $\mu_t$ for subcritical BRWs without any assumptions on the variance of random walk jumps and with $N\geq1$ sources of different positive intensities was found.

For $\beta\leq \beta_{c}$, the growth of the moments $\mu_t$ and
$\mu_{t}(y)$ of the particle numbers appears to be irregular with respect
to the moment number $n$ which means that the behavior
of the particle numbers $\mu_t$ and $\mu_{t}(y)$ as $t\to \infty$ differs appreciably from
the behavior of their moments and the study of the BRW survival probabilities becomes relevant.

\section{Survival Probabilities}\label{sect:SA}
We review in what follows the results on symmetric BRWs which may find
use in the studies of reliability. Let us consider for simplicity a  BRW with one branching source located at the origin. The notion of a system reliability is
related with the concepts of the \emph{probability $Q(t,x,y):=P_{x}\{\mu_{t}(y)>0\}$ of presence of particles at the  point $y\in\mathbf{Z}^d$}
and the \emph{survival
probability $Q(t,x):=P_{x}\{\mu_{t}>0\}$ of the particle population}.
The function $Q(t,x)$ may be called \emph{the system reliability function}.  The
asymptotic behavior of $Q(t,x,0)$ describes the number of working elements.
That is why in \citep{Y10-AiT:e} consideration was given to the asymptotic behavior of the
survival probability $Q(t,x)$ of the particle population on the lattice and
the probability $Q(t,x,0)$ of presence of particles at the origin at time~$t$.

Denote by $Q(t):=Q(t,0)$, $q(t):=Q(t,0,0)$, and $p(t):=p(t,0,0)$.
The equations for $Q(t,x,y)$ and $Q(t,x)$ were established in
\citep{Y09DM:e}.  If $x=y=0$, then they are
as follows:
\begin{align}
q(t)&=p(t)-\int_{0}^{t}p(t-s)f(1-q(s))\,ds,\\
\label{E:Q}
Q(t)&=1-\int_{0}^{t}p(t-s)f(1-Q(s))\,ds.
\end{align}
Under the condition~(\ref{E:infindisp}) the equations for $Q(t,x,y)$ and $Q(t,x)$ have the same representation as in the case of a finite variance of random walk jumps~(\ref{E:findisp}).
We do not give related the proof here because one can realize it by the proof scheme for  BRWs with a finite variance of jumps (see, Theorem~3 in \citep{Y09DM:e}). But in virtue of the difference in recurrence criteria between the BRWs under the conditions~\eqref{E:findisp} and~\eqref{E:infindisp}, we obtain another limit theorems for heavy-tailed BRWs in contrast to theorems established in \citep{Y09DM:e,Y10-CS:e}  and \citep{Y10-AiT:e} for BRWs with a finite variance of random walk jumps.

\section{Survival Probabilities for Heavy-Tailed BRWs}\label{sect:SPHTBRW}

The proofs of the theorems on the asymptotic behavior of the survival probability~\eqref{E:Q} for the heavy-tailed case~\eqref{E:infindisp} will be based on some statements about asymptotic behavior of the function $p(t,0,0)-p(t,x,0)$, $x\in\mathbf{Z}^d$, as $t\to\infty$; properties of the infinitesimal generation function  $f(u):=f(u,0)$; properties of the Laplace generation function $F(z,t,x)$ and ge\-ne\-ral\-izations of the theorems on a limit behavior of the survival probability for the case~\eqref{E:findisp} of finite variance of jumps. Let us prove these statements.

\begin{theorem}
For each $d\geq1$, $\alpha\in(0,2)$, and $x\in\mathbf{Z}^d$ the following asymptotic relation holds:
\begin{equation}\label{eq2114}
p(t,0,0) - p(t,x,0) \sim \frac{\tilde\gamma_{d,\alpha}(x)}{t^{\frac{d+2}{\alpha}}},\quad  t\to\infty,
\end{equation}
where
\[
\tilde \gamma_{d,\alpha}(x) = \frac{1}{2(2\pi)^d}\int\limits_{\mathbf{R}^d}\langle \omega,x\rangle^2 e^{-\eta\left(\frac{\omega}{|\omega|}\right)|\omega|^{\alpha}}\,d\omega
\]
and $\eta(u)$ is a continuous function on the unit sphere satisfying $0<\eta_{*}\leq \eta(u)\leq\eta^{*}<\infty$ for $u\in\mathbf{S}^{d-1}$.
\end{theorem}

\begin{proof}
Denote the Fourier transform of the operator $\mathcal{A}$ by $\phi(\theta)$, that is
\[
\phi(\theta)=\sum_{x\in\mathbf{Z}^d}a(x)e^{i\langle x,\theta\rangle},\quad \theta\in[-\pi,\pi]^d,
\]
where $\langle\cdot,\cdot\rangle$ is the Euclidean inner product in $\mathbf{R}^d$.

From, e.g., \citep{Y13-CommStat:e}, the transition probability for each $x,y\in\mathbf{Z}^d$, $t\geq0$, can be expressed in the form
\[
p(t,x,y) = \frac{1}{(2\pi)^d}\int\limits_{[-\pi,\pi]^d}e^{i\langle \theta,y-x\rangle}\,e^{\phi(\theta)t}d\theta.
\]
Due to symmetry of $a(x)$ on $\mathbf{Z}^d$, this implies
\begin{equation}\label{eq2117}
p(t,0,0) - p(t,x,0) = \frac{1}{(2\pi)^d}\int\limits_{[-\pi,\pi]^d}(1-\cos\langle \theta ,x\rangle)\,e^{\phi(\theta)t}d\theta.
\end{equation}
We can represent the function $1-\cos\langle\theta,x\rangle$ as
\[
1-\cos\langle \theta,x\rangle = \frac{1}2\langle \theta,x\rangle^2 + \nu(\theta,x),
\]
where $|\nu(\theta,x)|\leq c |x|^3|\theta|^3$ for some $c<\infty$.

By the definition of $\phi(\theta)$ and symmetry of $a(\cdot)$, the following equality holds
\[
\phi(\theta)=\sum_{x\in\mathbf{Z}^d}a(x)\cos\langle\theta,x\rangle,
\]
where the series $\sum_{x}a(x)$ is absolutely convergent, the functions $\cos\langle\theta,x\rangle$ are continuous and uniformly bounded  by unit in $\theta\in[-\pi,\pi]^d$ and $x\in\mathbf{Z}^d$. Hence, the function $\phi(\theta)$ is continuous on $[-\pi,\pi]^{d}$. Then for any fixed $t\geq0$ and $x\in\mathbf{Z}^{d}$ each of the functions $(1-\cos\langle\theta,x\rangle)e^{\phi(\theta)t}$, $\frac{1}2\langle\theta,x\rangle^2e^{\phi(\theta)t}$ and $\nu(\theta,x)e^{\phi(\theta)t}$ is integrable on the hypercube $[-\pi,\pi]^d$.

As shown in~\citep{Koz:IJARM16}, the function $\phi(\theta)$ has the following asymptotics
\begin{equation}\label{eq2118}
\phi(\theta) \sim -\eta\left(\frac{\theta}{|\theta|}\right)|\theta|^{\alpha},\quad |\theta|\to0,
\end{equation}
where $\eta(u)$ is a continuous function on the unit sphere satisfying $0<\eta_{*}\leq \eta(u)\leq\eta^{*}<\infty$ for $u\in\mathbf{S}^{d-1}$. To use the homogeneity of the function in the right-hand side of \eqref{eq2118} and to reduce the variable~$t$ in the exponent of the integral \eqref{eq2117}, we replace $\theta=\omega t^{-\frac{1}{\alpha}}$. Then for $p(t,0,0)-p(t,x,0)$ the following representation is valid
\begin{equation}\label{eq1}
\frac{1}{2(2\pi)^d}\ t^{-\frac{d+2}{\alpha}}L_1(t) + \frac{1}{(2\pi)^d}\ t^{-\frac{d+2}{\alpha}} L_2(t),
\end{equation}
where the functions $L_1(t)$ and $L_{2}(t)$ are as follows:
\begin{align*}
L_1(t) &= \int\limits_{[-\pi t^{\frac{1}{\alpha}},\pi t^{\frac{1}{\alpha}}]^d}\langle \omega ,x\rangle^2 e^{\phi\left(t^{-\frac{1}{\alpha}}\omega\right)t}\,d\omega,\\
L_2(t) &= \int\limits_{[-\pi t^{\frac{1}{\alpha}},\pi t^{\frac{1}{\alpha}}]^d}t^{\frac{2}{\alpha}}\nu(t^{-\frac{1}{\alpha}}\omega,x)e^{\phi\left(t^{-\frac{1}{\alpha}}\omega\right)t}\,d\omega.
\end{align*}

Let us fix an arbitrary $\varepsilon>0$. According to~\eqref{eq2118}, choose a $\rho_{\varepsilon}<\pi$ such that in the set
\[
\mathcal{O}\left(\rho_{\varepsilon}t^{\frac{1}{\alpha}}\right):=\left\{w\in\mathbf{R}^d:~ |w|<\rho_{\varepsilon}t^{\frac{1}{\alpha}}\right\}
\]
the relations
\begin{equation}\label{eq2}
-(1+\varepsilon)\,\eta\left(\frac{w}{|w|}\right)|t^{-\frac{1}{\alpha}}w|^{\alpha}\leq \phi\left(t^{-\frac{1}{\alpha}}w\right) \leq -(1-\varepsilon)\,\eta\left(\frac{w}{|w|}\right)|t^{-\frac{1}{\alpha}}w|^{\alpha}
\end{equation}
 hold. Then the function $L_1(t)$ can be represented as the sum $L_{1,1}(\varepsilon,t) + L_{1,2}(\varepsilon,t)$, where
\begin{align*}
L_{1,1}(\varepsilon,t) &= \int\limits_{\mathcal{O}\left(\rho_{\varepsilon}t^{\frac{1}{\alpha}}\right)}\langle \omega ,x\rangle^2\,e^{\phi\left(t^{-\frac{1}{\alpha}}\omega\right)t}\,d\omega,\\
L_{1,2}(\varepsilon,t) &= \int\limits_{[-\pi t^{\frac{1}{\alpha}},\pi t^{\frac{1}{\alpha}}]^d\setminus \mathcal{O}\left(\rho_{\varepsilon}t^{\frac{1}{\alpha}}\right)}\langle \omega ,x\rangle^2\,e^{\phi\left(t^{-\frac{1}{\alpha}}\omega\right)t}\,d\omega.
\end{align*}

Consider the function $L_{1,2}(\varepsilon,t)$. Note  that according to Lemma~2.1.3 in~\citep{YarBRW:e}, whose proof remains valid for the heavy-tailed case, for each $\theta\in[-\pi,\pi]^d$ the estimate $\phi(\theta)\leq -(\gamma/d)|\theta|^2$ with some $\gamma>0$ holds. Then
\[
L_{1,2}(\varepsilon,t) \leq |x|^2\ e^{-\frac{\gamma}{d}|\rho_{\varepsilon}|^2\,t} \int\limits_{[-\pi t^{\frac{1}{\alpha}},\pi t^{\frac{1}{\alpha}}]^d \setminus \mathcal{O}\left(\rho_{\varepsilon}t^{\frac{1}{\alpha}}\right)} |\omega|^2\,d\omega \leq c_d(x)\ t^{\frac{d+2}{\alpha}}\,e^{-\frac{\gamma}{d}|\rho_{\varepsilon}|^2\,t},
\]
where
$c_d(x)=(2^d/3)d |x|^2\pi^{d+2}$. Hence the function $L_{1,2}(\varepsilon,t)$ decreases exponentially as $t\to\infty$.

Now we consider the function~$L_{1,1}(\varepsilon,t)$. From \eqref{eq2} it follows that
\begin{equation}\label{eq4}
Q(\varepsilon,1+\varepsilon,t) \leq L_{1,1}(\varepsilon,t) \leq Q(\varepsilon,1-\varepsilon,t)\quad\text{for~} t\geq0,
\end{equation}
where
\[
Q(\varepsilon,\mu,t) = \int\limits_{\mathcal{O}\left(\rho_{\varepsilon}t^{\frac{1}{\alpha}}\right)}\langle \omega ,x\rangle^2 e^{-\mu\,\eta\left(\frac{w}{|w|}\right)|\omega|^{\alpha}}\,d\omega.
\]
Note that for the function $Q(\varepsilon,\mu,t)$ the integration domain depends on $t$, but   the integrand does not. To find the $\lim_{t\to\infty}Q(\varepsilon,\mu,t)$, we switch to the generalized polar coordinates and use the designations $s(\varphi)$ and $\eta(\varphi)$, $\varphi\in\mathbf{R}^{d-1}$, respectively,
for the functions $\bigl\langle \frac{\omega}{|\omega|},x \bigr\rangle^2$ and $\eta\bigl(\frac{\omega}{|\omega|}\bigr)$, because they are independent of the polar radius. Then
\begin{multline*}
\lim_{t\to\infty}Q(\varepsilon,\mu,t) = \int\limits_{\mathbf{R}^d} \langle \omega/|\omega| ,x\rangle^2\,|\omega|^2\ e^{-\mu\,\eta\left(\frac{w}{|w|}\right)|\omega|^{\alpha}}\,d\omega \\
= \int_0^{\infty}\int_0^{2\pi}\int_0^{\pi}\cdots\int_0^{\pi} s(\varphi)\,r^2e^{-\mu\,\eta(\varphi)r^{\alpha}}r^{d-1}\ \sin\varphi_2\cdots(\sin\varphi_{d-1})^{d-2}\,dr\,d\varphi_1\cdots d\varphi_{d-1} \\
= \tilde \gamma_{d,\alpha}(x)\,\mu^{-\frac{d+2}{\alpha}}
\end{multline*}
where
\begin{multline*}
\tilde \gamma_{d,\alpha}(x)  = \frac{1}{\alpha}\,\Gamma\left(\frac{d+2}{\alpha}\right)\int_0^{2\pi}\int_0^{\pi}\cdots\int_0^{\pi}\frac{s(\varphi)}{(\eta(\varphi))^{\frac{d+2}{\alpha}}} \sin\varphi_2\cdots(\sin\varphi_{d-1})^{d-2}\,d\varphi_1\cdots d\varphi_{d-1} \\
= \int\limits_{\mathbf{R}^d}\langle \omega ,x\rangle^2 e^{-\eta\left(\frac{w}{|w|}\right)|\omega|^{\alpha}}\,d\omega.
\end{multline*}
From \eqref{eq4} it follows
\[
\tilde \gamma_{d,\alpha}(x)\,(1+\varepsilon)^{-\frac{d+2}{\alpha}} \leq \lim_{t\to\infty}L_{1,1}(\varepsilon,t) \leq \tilde \gamma_{d,\alpha}(x)\,(1-\varepsilon)^{-\frac{d+2}{\alpha}},
\]
and, due to arbitrariness of $\varepsilon$, we have $\lim_{t\to\infty}L_{1,1}(\varepsilon,t) = \tilde \gamma_{d,\alpha}(x)$.
Thus
\[
L_1(t) \sim \tilde \gamma_{d,\alpha}(x),\quad  t\to\infty.
\]

It remains to consider the estimate
\[
L_2(t) \leq c\,|x|^3 t^{-\frac{1}{\alpha}}\int\limits_{[-\pi t^{\frac{1}{\alpha}},\pi t^{\frac{1}{\alpha}}]^d}|\omega|^3\,e^{\phi(t^{-\frac{1}{\alpha}}\omega)t}\,d\omega,
\]
therefore, the function $L_2(t)$ decreases faster than the function $L_1(t)$ for $t\to\infty$.

Then from the representation of $p(t,0,0)-p(t,x,0)$ in the form of~\eqref{eq1}, we find the asymptotics~\eqref{eq2114}.
\end{proof}

Let us prove the following lemma about properties of the infinitesimal generation function.

\begin{lemma}\label{f_nonneg}
Let $b_n\geq0$ for $n\not=1$, $b_1<0$ and $\sum_nb_n=0$, $\sum_nnb_n<\infty$. Assume that, for each $r>1$, the derivatives $f^{(r)}(u)$ for $u=1$ are defined and finite.
Then
\begin{itemize}
\item[\rm(i)]  if $\beta:=f'(1)\leq 0$, then  the function $f(u)$ is strictly positive for  $u\in[0,1)$ and $f(1)=0$, the function $f'(u)$ is non-positive for all $u\in[0,1]$,  and the functions $f^{(r)}(u)$, $r\geq2$, are  non-negative for the same values of $u$;
\item[\rm(ii)]  if $\beta>0$, then there exists $u_{*}\in[0,1)$ such that the function $f(u)$ is strictly positive for $u\in(0,u_{*})$ (if $u_{*}>0$), it is strictly negative for $u\in(u_{*},1)$ and $f(u_{*})=f(1)=0$, the function $f'(u)$ has exactly one sign change on the interval $u\in[0,1]$.
\end{itemize}
\end{lemma}

\begin{proof}
Note that $\beta=\sum_nnb_n$.
Then
\[
b_0 = -b_1 - \sum_{n=2}^{\infty}b_n = -\beta +\sum_{n=2}^{\infty}nb_n - \sum_{n=2}^{\infty}b_n = -\beta + \sum_{n=2}^{\infty}(n-1)b_n,
\]
and therefore
\begin{equation}\label{f_beta}
f(u) = -\beta(1-u) + \sum_{n=2}^{\infty}b_ng_n(u),
\end{equation}
where $g_n(u):=u^n-nu+n-1$. For each $n\geq2$ and $u\in[0,1]$ the following estimates
\begin{align*}
g_n(u) &= (1-u)\left(n-\sum_{k=0}^{n-1}u^k\right)\geq0,\\
g'_n(u) &= n (u^{n-1}-1)\leq0,\\
g_n^{(r)}(u) &=  \frac{n\,!}{(n-r)\,!}\,b_nu^{n-r}\geq0,\quad r\geq2,
\end{align*}
are valid. Hence, taking into account that $\beta\leq0$, from~\eqref{f_beta} we obtain the statements of Lemma~\ref{f_nonneg}.
\end{proof}

Let us present a statement about the properties of the Laplace generating function $F(z,t,x)$, which complements the statement of Lemma~2 from~\citep{Y09DM:e}.

\begin{lemma}\label{Fmonot}
The function $F(z,t,x)$ is monotone in $t$ for fixed $z\geq0$ and $x\in\mathbf{Z}^d$: it does not decrease if $f(e^{-z})>0$; it does not increase if $f(e^{-z})<0$; $F(z,t,x)\equiv e^{-z}$ if $f(e^{-z})=0$. Consequently, the probability $Q(t,x)$ of the continuation of the process does not increase in $t$ for any fixed $x\in\mathbf{Z}^d$.
\end{lemma}

The proof of this lemma is based on the theorem on positive solutions for differential equations in $\ell^{\infty}(\mathbf{Z}^d)$ with off-diagonal positive right-hand side (see, e.g., \citep{MAK:e}).

Let us generalize the assertion of Lemma~9 from~\citep{Y09DM:e} using an appropriate proof scheme, for the case of the recurrent BRWs with the branching source intensity $\beta\leq0$ and with no assumptions on the variance of jumps.

\begin{lemma}\label{CrSubcr}
If BRW is recurrent, that is $G_0=\infty$, and $\beta\leq0$, then
\[
\lim_{t\to\infty}F(z,t,x)=1,\qquad \lim_{t\to\infty}Q(t,x)=0.
\]
\end{lemma}

\begin{proof}
By Lemma~\ref{f_nonneg} the condition $\beta\leq0$ implies $f(e^{-z})>0$ for each fixed $z>0$ and $f(e^{-z})=0$ for $z=0$. Hence, by Lemma~\ref{Fmonot} the function $F(z,t,x)$ does not decrease in $t$ for each fixed $z>0$ and $F(z,t,x)\equiv1$ for $z=0$.
Then, according to the reasoning of Lemma~9 in~\citep{Y09DM:e}, there exists $c_0(z,x)\geq0$ such that $\lim_{t\to\infty}F(z,t,x)=1-c_0(z,x)\leq1$. Note  that $F(0,t,x)\equiv1$. Assume $c_0(z,x)>0$ for $z>0$,  then we get
\begin{equation}
\label{eqc0}
1 - c_0(z,x) - e^{-z} = \lim_{t\to\infty}\int_0^tp(t-s,x,0)f(F(z,s,0))\,ds.
\end{equation}
Now, from Lemma~2 in~\citep{Y10-CS:e} it follows that
\[
f(F(z,t,0)) =
    \begin{cases}
        \frac{f''(1)}{2}\,c_0^2(z,0)+o(t) & \text{for}\ \beta=0, \\
        -\beta\,c_0(z,0)+o(t) & \text{for}\ \beta<0,
     \end{cases}
\]
as $t\to\infty$, and therefore the function $f(F(z,t,0))$  for each fixed $z>0$ has a finite nonzero limit as $t\to\infty$. Then the left-hand side of~\eqref{eqc0} is finite, whereas according to $G_0=\infty$, the term in the right-hand side is infinite. Hence, $c_0(z,x)\equiv0$ and then $\lim_{t\to\infty}Q(t,x)=0$.
\end{proof}

Under condition \eqref{E:findisp} the proof of the above lemma for dimensions $d=1,2$ is given in Lemma~3 of~\citep{Y10-CS:e}, whereas for dimensions $d\geq3$ the appropriate result was only announced. Let us give a proof of its generalization to the case of the transient
BRWs.

Previously we prove the following statement.

\begin{lemma}
\label{Theorem547}
Assume the functions $\varphi(t)$ and $\chi(t)$ are continuous for $t\geq0$, $\varphi(t)\to\varphi_0$ as $t\to\infty$, $\chi(t)\geq0$, $\chi(t)\not\equiv0$, and
\begin{equation}
\frac{\int_{t-T}^{t}\chi(s)\,ds}{\int_0^t\chi(s)\,ds}\to0\quad \text{as}\ t\to\infty\label{eq5123}
\end{equation}
for sufficiently large $T$. Then
\[
\frac{\int_{0}^{t}\varphi(t-s)\chi(s)\,ds}{\int_0^t\chi(s)\,ds}\to\varphi_0\quad \text{as}\ t\to\infty.
\]
In particular, the limit relation~\eqref{eq5123} holds if $\int_0^{\infty}\chi(s)\,ds<\infty$.
\end{lemma}

\begin{proof}
Let us fix some $\varepsilon>0$. Choose a $T$ such that $|\varphi(t)-\varphi_0|\leq\varepsilon$ for $t\geq T$.  Note  that $\int_0^t\chi(s)\,ds>0$ for all sufficiently large $t$, and for those $t$ the following function is defined
\[
W(t) := \frac{\int_0^t\varphi(t-s)\chi(s)\,ds}{\int_0^t\chi(s)\,ds} - \varphi_0 = \frac{\int_0^t(\varphi(t-s)-\varphi_0)\chi(s)\,ds}{\int_0^t\chi(s)\,ds}.
\]
Represent $W(t)$ as the sum of $W_1(t)$ and $W_2(t)$, where
\[
W_1(t):=\frac{\int_0^{t-T}(\varphi(t-s)-\varphi_0)\chi(s)\,ds}{\int_0^t\chi(s)\,ds},\qquad
W_2(t):=\frac{\int_{t-T}^{t}(\varphi(t-s)-\varphi_0)\chi(s)\,ds}{\int_0^t\chi(s)\,ds}.
\]
According to the choice of $T$ we have $|W_1(t)| \leq \varepsilon$. Also, the estimate is valid
\[
|W_2(t)|\leq \sup_{t\geq0}|\varphi(t)-\varphi_0|\cdot\frac{\int_{t-T}^{t}\chi(s)\,ds}{\int_0^t\chi(s)\,ds},
\]
and, by the lemma conditions, $W_2(t)\to0$ as $t\to\infty$. Hence, $\limsup_{t\to\infty}|W(t)|\leq \varepsilon$, whence, due to the arbitrariness of $\varepsilon$, the assertion of the lemma follows.
\end{proof}

Now we will generalize the result of Lemma~3 from~\citep{Y10-CS:e}  for the case of transient BRWs on $\mathbf{Z}^d, d\geq3$, without any assumptions on the variance of jumps and on the branching source intensity $\beta$.

\begin{lemma}
\label{limQfin}
For a transient BRW, that is $G_0<\infty$, for every $z\geq0$ and $x\in\mathbf{Z}^d$ the following equalities hold
\[
\lim_{t\to\infty}F(z,t,x)=1-c_d(x),\qquad \lim_{t\to\infty}Q(t,x)= c_d(x),
\]
where $\lim_{z\to\infty}c_d(z,x) = c_d(x)>0$ and $ c_d(z,x)$ is the least non-negative root of the equation
\begin{equation}\label{Sk29}
1 - c_d(z,x) - e^{-z} = G_0(x,0)f(1-c_d(z,0)).
\end{equation}
\end{lemma}

\begin{proof}
The function $F(z,t,x)$ satisfies the equality
\begin{equation}\label{Feq}
F(z,t,x) = e^{-z} + \int_0^tp(t-s,x,0)f(F(z;s,0))\,ds
\end{equation}
(see Theorem~2.7 in~\citep{AB00:e}).

Under the lemma conditions, the value $G_0(x,0)=\int_0^{\infty}p(s,x,0)\,ds$ is finite for each $x\in\mathbf{Z}^d$. By Lemma~\ref{Fmonot}, the function $F(z,t,x)$ is monotonous in $t$ for each fixed $z\geq0$, $x\in\mathbf{Z}^d$, $d\geq1$. Hence, due to Lemma~\ref{f_nonneg} the function $f(F(z,t,0))$ has a finite limit as $t\to\infty$. Then, passing to the limit as $t\to\infty$ in  equility~\eqref{Feq}, by Lemma~\ref{Theorem547} we get that the value $c_d(z,x)$ satisfies the relation~\eqref{Sk29}.

Show that the value $c_d(z,x)$ is strictly positive. Indeed, if the value $c_d(z,0)$ were equal to zero, then the term in the right-hand side of equation~\eqref{Sk29} would vanish, while the term in the left-hand side would be nonzero. Hence, $c_d(z,0)>0$. Due to the estimate $G_0(x,0)\leq G_0$ (see, e.g., representation~(2.11) in~\citep{Y18-TPA:e}) and equation~\eqref{Sk29}, we get $1-c_d(z,x)-e^{-z}\leq G_0f(1-c_d(z,0))$. Therefore, $1-c_d(z,x)-e^{-z} \leq 1-c_d(z,0)-e^{-z}$, and hence
\[
0<c_d(z,0)\leq c_d(z,x),
\]
that completes the proof of the lemma.
\end{proof}

Further we will consider the statements for BRWs with heavy tails~\eqref{E:infindisp}.

\begin{theorem}
\label{QtotalHT}
If  $\beta \leq \beta_c$ then for every $z>0$ and $x\in \mathbf{Z}^d$ one has
\[
\lim_{t\to\infty}F(z,t,x)=1,\qquad \lim_{t\to\infty}Q(t,x)=0
\]
for $d=1$, $\alpha\in[1,2)$, and
\[
\lim_{t\to\infty}F(z,t,x)=1-\tilde c_{d,\alpha}(x),\qquad \lim_{t\to\infty}Q(t,x)=\tilde c_{d,\alpha}(x)
\]
for $d=1$, $\alpha\in(0,1)$ and for $d\geq2$, $\alpha\in(0,2)$, where
$\lim_{z\to\infty}\tilde c_{d,\alpha}(z,x)=\tilde c_{d,\alpha}(x)>0$
and $\tilde c_{d,\alpha}(z,x)$ is the least non-negative root of the equation
\[
1 - \tilde c_{d,\alpha}(z,x) - e^{-z} = G_0(x,0)f(1-\tilde c_{d,\alpha}(z,0)).
\]
\end{theorem}

\begin{proof}
In the case $\beta\leq\beta_c$, $d=1$, $\alpha\in[1,2)$ the BRW under consideration is recurrent and $\beta_c=0$ (see Theorem~1 in~\citep{Y16-MCAP:e}). Therefore, the appropriate statement follows from Lemma~\ref{CrSubcr}.

The statement for $d=1$, $\alpha\in(0,1)$ and for $d\geq2$, $\alpha\in(0,2)$ can be proved by the same scheme as in Lemma~\ref{limQfin}, because under these conditions BRW with heavy tails is transient, that is $G_0<\infty$.
\end{proof}

\begin{theorem}\label{ThQAsymp}
If $\beta\leq\beta_c$ then the survival probabilities $Q(t,x)$ have the following asymptotics as $t\to\infty$:
\[
Q(t,x)\sim
\begin{cases}
\tilde C_{d,\alpha}(x)v(t)\quad&\text{for~}\beta=\beta_{c},\\
\tilde K_{d,\alpha}(x)u(t)\quad&\text{for~}\beta<\beta_{c},
\end{cases}
\]
where the functions $v(t)$ and $u(t)$ are of the form:
\[
\begin{array}{lll}
  v(t)=t^{\frac{1-\alpha}{2\alpha}}, &\quad u(t)=t^{\frac{1-\alpha}{\alpha}} &\qquad\text{for}\quad d=1,~\alpha\in(1,2),\\
  v(t)=(\ln t)^{-\frac{1}{2}}, &\quad u(t)=(\ln t)^{-1}&\qquad\text{for}\quad d=1,~\alpha=1,\\
  v(t)\equiv 1, &\quad u(t)\equiv 1&\qquad\text{for}\quad d\geq 2,~\alpha\in(0,2)\\
\end{array}
\]
and
\[
\tilde C_{1,\alpha}(x)=\sqrt{2}\left(f''(1)\Gamma(\alpha^{-1})\gamma_{1,\alpha}\right)^{-\frac{1}{2}}, \quad \tilde K_{1,\alpha}(x)=\left(-\beta\,\Gamma(\alpha^{-1})\gamma_{1,\alpha}\right)^{-1}
\]
for $d=1$, $\alpha\in(1,2)$, whereas
\[
\tilde C_{1,1}(x)=\sqrt{2}\left(f''(1)\gamma_{1,1}\right)^{-\frac{1}{2}}, \quad  \tilde K_{1,1}(x)=\left(-\beta\,\gamma_{1,1}\right)^{-1}
\]
for $d=1$, $\alpha=1$.

Both the functions $\tilde C_{d,\alpha}(x)$ and $\tilde K_{d,\alpha}(x)$ for every $x\in\mathbf{Z^{d}}$, $d\geq 2$ and $\alpha\in(0,2)$  are determined by
$1-\beta_{c}G_{0}(x,0)\left(1-\tilde c_{d,\alpha}(0)\right)$, where
$\tilde c_{d,\alpha}(0)$ is the non-negative root of the equation
\[
\beta_{c}(1-\tilde c_{d,\alpha}(0))=f(1-\tilde c_{d,\alpha}(0)).
\]
\end{theorem}

\begin{proof}
The proof of the statements for recurrent BRW, that is for  $d=1$, $\alpha\in[1,2)$ is based on the proof scheme of Theorem~11 from \citep{Y10-CS:e} about the survival probability asymptotic behavior for BRW with a finite variance of jumps.

According to equation~(23) of Theorem~3 from~\citep{Y09DM:e}, the survival probability  $Q(t,x)$ satisfies to the equality
\begin{equation}\label{eqQtotal}
Q(t,x) = 1 - \int_0^tp(t-s,x,0)f(1-Q(s,0))\,ds.
\end{equation}
Let us consider the case $x=0$. Then the Laplace transform of the equation~\eqref{eqQtotal} has the following form:
\[
\widehat{1-Q}(\lambda) = G_{\lambda}\widehat{f(1-Q)}(\lambda).
\]
Theorem~\ref{QtotalHT} implies that $\widehat{1-Q} \sim \lambda^{-1}$ as $\lambda\to0$. Due to Theorem~1 from~\citep{Y18-TPA:e}, the asymptotics as $\lambda\to0$ of the appropriate Green's functions is as follows:
\[
G_{\lambda} \sim
    \begin{cases}
        \gamma_{1,\alpha}\,\lambda^{-\frac{1-\alpha}{\alpha}}, & \text{if~} d=1,~\alpha\in(1,2), \\
        -\gamma_{1,1}\ln\lambda, & \text{if~} d=1,~\alpha=1,
    \end{cases}
\]
where $\gamma_{1,\alpha}$ and $\alpha\!\in\![1,2)$ is some positive constant. Then, as $\lambda\to0$, we have
\[
\widehat{f(1-Q)}(\lambda) \sim
    \begin{cases}
        \gamma^{-1}_{1,\alpha}\,\lambda^{-\frac{1}{\alpha}}, & \text{if~} d=1,~\alpha\in(1,2), \\
        -\gamma^{-1}_{1,1}(\lambda\ln\!\lambda)^{-1}, & \text{if~} d=1,~\alpha=1.
    \end{cases}
\]

For $\beta\leq0$ the function $f(1-Q(t))$ is monotonous in $t$ due to the monotony of the functions $Q(t)$ (see Lemma~2 in~\citep{Y09DM:e}) and $f(u)$, $u\in[0,1]$ (see Lemma~\ref{f_nonneg}). Therefore, applying the Tauberian theorem for densities (see Theorem~4 Ch. XIII in~\citep{Feller}) as $t\to\infty$ we get the asymptotic equality
\[
f(1-Q(t)) \sim
    \begin{cases}
        (\Gamma(\alpha^{-1})\gamma_{1,\alpha})^{-1}\,t^{\frac{1-\alpha}{\alpha}} & \text{for~} d=1,~\alpha\in(1,2), \\
        \gamma^{-1}_{1,1}(\ln t)^{-1} & \text{for~} d=1,~\alpha=1.
    \end{cases}
\]
From Theorem~\ref{QtotalHT} it follows that  $\lim_{t\to\infty}(1-Q(t))=1$. Therefore, for $f(1-Q(t))$ as $t\to\infty$ the asymptotic behavior of the function $f(u)$ as $u\to1$ from Lemma~2 in~\citep{Y10-CS:e} can be used. Then for $t\to\infty$ we obtain the assertion for the critical case $\beta=\beta_c$:
\begin{equation}\label{QasympCr}
Q(t) \sim
    \begin{cases}
       \sqrt{2}\left(f''(1)\,\Gamma(\alpha^{-1})\,\gamma_{1,\alpha}\right)^{-\frac{1}{2}}\,t^{\frac{1-\alpha}{2\alpha}} & \text{for~} d=1,~\alpha\in(1,2), \\
        \sqrt{2}\left(f''(1)\,\gamma_{1,1}\right)^{-\frac{1}{2}}(\ln t)^{-\frac{1}{2}} & \text{for~} d=1,~\alpha=1,
    \end{cases}
\end{equation}
and for the subcritical case $\beta<\beta_c$:
\begin{equation}\label{QasympSub}
    Q(t) \sim
    \begin{cases}
       \left(-\beta\,\Gamma(\alpha^{-1})\,\gamma_{1,\alpha}\right)^{-1}\,t^{\frac{1-\alpha}{\alpha}} & \text{for~} d=1,~\alpha\in(1,2), \\
        \left(-\beta\,\gamma_{1,1}\right)^{-1}\,(\ln t)^{-1} & \text{for~} d=1,~\alpha=1.
    \end{cases}
\end{equation}

Now we will find the corresponding asymptotics for arbitrary $x\in\mathbf{Z}^d$. Note  that the first term of the asymptotic expansion of Green's function $G_{\lambda}(x,0)$ as $\lambda\to0$ is independent of $x$ (see Theorem~1 in~\citep{Y18-TPA:e}). This implies that the first term of the asymptotic expansion of the probability $Q(t,x)$ as $t\to0$ is independent of $x$. Hence, the representations~\eqref{QasympCr} and~\eqref{QasympSub} remain valid for arbitrary $x$.

Let us establish the required assertions for $d\geq2$, $\alpha\in(0,2)$. By Theorem~\ref{QtotalHT} we have
\[
\lim_{t\to\infty}Q(t,x) = \tilde c_{d,\alpha}(x),
\]
hence, the functions $\tilde C_{d,\alpha}(x)$, $\tilde K_{d,\alpha}(x)$ coincide with $\tilde c_{d,\alpha}(x)$ for each $x\in\mathbf{Z}^d$.
\end{proof}

\section{Conclusion}\label{sect:Conclusion}

Let us compare the behavior of survival probabilities of the particle population for critical and subcritical BRWs with a finite variance of jumps and heavy tails. The limit of  survival probabilities as $t\to\infty$ equals zero for recurrent BRWs and is strictly positive for transient BRWs, which is determined by the Green's function behavior. Therefore, this limit depends on the assumptions~\eqref{E:findisp} and \eqref{E:infindisp}.
On the lattice $\mathbf{Z}^d$, $d\geq3$, a BRW is transient under each assumption~\eqref{E:findisp} or \eqref{E:infindisp}. Therefore, the particle population survival probability for critical and subcritical BRWs tends to a positive constant. On $\mathbf{Z}^2$ the survival probabilities of the heavy-tailed BRW tend to a positive constant for  any $\alpha\in(0,2)$ in both the critical and subcritical cases, while the survival probabilities of the BRW with finite variance of jumps decreases asymptotically to zero as $(\ln t)^{-1}$ for  subcritical cases and as $(\ln t)^{-1/2}$ for critical ones. Thus, on $\mathbf{Z}^2$ the survival probabilities of a BRW with heavy tails decrease slower than the survival probabilities of a BRW with a finite variance of jumps.

Now we consider the asymptotic behavior of the functions $Q(t,x)$ with $x\in\mathbf{Z}$ for $t\to\infty$.
If the variance of jumps is finite then $\beta_{c}=0$ for $d=1$ and we obtain
\[
Q(t,x) \sim
    \begin{cases}
        C_1(x)\,t^{-\frac{1}{4}} & \text{for~} \beta=\beta_c, \\
        K_1(x)\,t^{-\frac{1}{2}} & \text{for~} \beta<\beta_c,
    \end{cases}
\]
where $C_1(x)$, $K_1(x)$ are some positive constants (see~Theorem~1 in \citep{Y10-CS:e}).
If a BRW has heavy tails~\eqref{E:infindisp} then $\beta_{c}=0$ for $\alpha\in [1,2)$ and $\beta_{c}>0$ for $\alpha\in (0,1)$. In this case we get more complex asymptotic behavior for the survival probability $Q(t,x)$ of the particle population
\[
Q(t,x) \sim
    \begin{cases}
         \tilde C_{1,\alpha}(x)  &\text{for~} d=1,\ \alpha\in(0,1)\text{~and~}\beta=\beta_c,\\
        \tilde C_{1,1}(x)\,(\ln t)^{-\frac{1}{2}} & \text{for~} d=1,\ \alpha=1\text{~and~}\beta=\beta_c,\\
        \tilde C_{1,\alpha}(x)\,t^{-\frac{1}{4}} & \text{for~}d=1,\ \alpha\in(1,2)\text{~and~} \beta=\beta_c,\\
         \tilde K_{1,\alpha}(x)  &\text{for~} d=1,\ \alpha\in(0,1)\text{~and~}\beta<\beta_c,\\
        \tilde K_{1,1}(x)\,(\ln t)^{-1} & \text{for~} d=1,\ \alpha=1\text{~and~}\beta<\beta_c,\\
        \tilde K_{1,\alpha}(x)\,t^{-\frac{1}{2}} & \text{for~} d=1,\ \alpha\in(1,2)\text{~and~}\beta<\beta_c,
    \end{cases}
\]
where $\tilde C_{1,\alpha}(x)$, $\tilde K_{1,\alpha}(x)$ are some positive constants (see~Theorem~\ref{ThQAsymp}).

\section{Acknowledgements}
The authors are supported by the Russian Foundation for Basic Research, Project number \mbox{17-01-00468}.


\end{document}